\documentclass{article}

\usepackage{amsmath, amsfonts, amsthm, amssymb, mathrsfs}

\usepackage{mathtools}
\usepackage{microtype}

\usepackage{float}
\usepackage{enumerate}
\usepackage{graphicx}

\usepackage[hidelinks]{hyperref}

\usepackage{multicol}

 \usepackage[all]{xy}

\usepackage{paralist}
\usepackage{geometry}
\def\Z{\mathbb{Z}}
\def\R{\mathbb{R}}
\def\C{\mathbb{C}}

\def\N{\mathcal{N}}
\def\H{\mathcal{H}}
\def\P{\mathcal{P}}

\newtheorem{theorem}{Theorem}[section]
\newtheorem*{theorem*}{Theorem}

\newtheorem{lemma}[theorem]{Lemma}
\newtheorem{proposition}[theorem]{Proposition}
\newtheorem{corollary}[theorem]{Corollary}

\theoremstyle{definition}
\newtheorem{problem}{Problem}

\newtheorem{definition}[theorem]{Definition}
\newtheorem{remark}[theorem]{Remark}
\newtheorem{example}[theorem]{Example}

\def\int{\operatorname{Int}}
\def\es{\operatorname{S}}
\def\de{\operatorname{D}}
\def\reeb{\mathcal{R}}

\def\ab{\operatorname{Ab}}
\def\geng#1{\langle #1 \rangle}

\def\corank{\operatorname{corank}}

\def\rank{\operatorname{rank}}

\def\cat{\operatorname{cat}}

\newcommand{\comment}[1]{}

\usepackage{tikz}
\usetikzlibrary{matrix,positioning,calc}
\usetikzlibrary{decorations.pathmorphing}
\usetikzlibrary{decorations.pathreplacing}

\tikzset{snake it/.style={-stealth,
        decoration={snake,
            amplitude = .4mm,
            segment length = 2mm,
            post length=0.9mm},decorate}}

\makeatletter
\def\@addpunct#1{%
    \relax\ifhmode
    \ifnum\spacefactor>\@m \else#1\fi
    \fi}
    \newcommand{\zz}[1]{}
\newcommand{\keywordsname}{$2020$ Mathematics Subject Classification}
\def\@setkeywords{%
    {\itshape \keywordsname.}\enspace \@keywords\@addpunct.}
\def\keywords#1{\def\@keywords{#1}}
\let\@keywords=\@empty
\g@addto@macro{\maketitle}{\begingroup%
    \let\@makefnmark\relax  \let\@thefnmark\relax%
    \ifx\@keywords\@mpty\else\@footnotetext{\@setkeywords}\fi%
    \endgroup}
\makeatletter

\keywords{Primary: 57M15; Secondary: 57K31, 20F05.  
    \\\indent\indent 
    Research supported by the grant Sheng 1 UMO-2018/30/Q/ST1/00228 of National Science Centre, Poland, \\
    \indent\indent and by the Slovenian Research and Innovation Agency program P1-0292 and grant J1-4001.
}

\newcommand{\address}{{ \bigskip

\footnotesize
    {\noindent\textsc{\L{}ukasz Patryk Michalak}\\ \\
   		Institute of Mathematics, Physics and Mechanics,  \\
    	Jadranska 19, 1000 Ljubljana, Slovenia,  \\ 
    	and \\ 
    	Faculty of Mathematics and Computer Science, \\
    	Adam Mickiewicz University, Pozna\'n, \\
    	ul.~Uniwersytetu Pozna\'nskiego 4, 61-614 Pozna\'n, Poland} \\
    	\\
    \textit{E-mail address:} \texttt{lukasz.michalak@amu.edu.pl}

}}

\date{}

\title{Reeb graph invariants of Morse functions \\ and $3$-manifold groups}

\author{\L{}ukasz Patryk Michalak}

\numberwithin{equation}{section}

\begin{document}

    \maketitle
    \begin{abstract}   	
    	\noindent In this work we are focused on the existence of Morse functions on a closed manifold~$M$ which are far from being ordered, i.e. whose Reeb graphs have positive first Betti number, especially the maximal possible, equals $\operatorname{corank}(\pi_1(M))$. In~the case of $3$-manifolds we describe the minimal number of critical points needed to construct such functions, which is related with the number of vertices of degree $2$ in Reeb graphs. We define a~new invariant of $3$-manifold groups and their presentations, and using Heegaard splittings we show its utility in determining occurrence of disordered Morse functions. In particular, the \emph{Freiheitssatz}, a~result for one-relator groups, allows us to calculate this invariant in the case of orientable circle-bundles over a surface, which provides an interesting example of the behaviour of Morse functions. 
    \end{abstract}

\section{Introduction}

Morse functions on a smooth closed manifold $M$ are a strong tool in the study of its topological properties. Each Morse function induces a handle decomposition of a manifold which also leads to its CW structure.
Moreover, every smooth closed manifold admits a Morse function and, without changing its critical points, it can be perturbed to be simple, i.e. on each critical level there is exactly one critical point. It gives a partition of a manifold on elementary cobordisms and together with rearrangement and cancellation theorems it constitutes a powerful technique in the proof of h-cobordism theorem of Smale (see \cite{Milnor}).
In the first step of the proof the order of consecutive critical points or handles is changed, obtaining an ordered Morse function in which the sequence of indices of critical points is non-decreasing as critical values increase.

In this work, we deal with the problem how disordered a Morse function on a given manifold can be. Its importance appears in the study of Reeb graphs. The Reeb graph $\reeb(f)$ of a Morse function $f\colon M \to \R$ is a graph obtained by contracting connected components of level sets of $f$ (G. Reeb \cite{Reeb}).  We~propose to measure disordering of a~Morse function by the homotopy type of its Reeb graph, which is determined by the first Betti number $\beta_1(\reeb(f))$ called the cycle rank of $\reeb(f)$. It is easy to observe that the Reeb graph of every ordered Morse function on a manifold of dimension $n \geq 3$ is a tree, so it has no cycles (\cite[Proposition 3.2]{Michalak-DCG}). On the other hand, we proved \cite[Theorem 5.2]{Michalak-DCG} that
the maximum possible cycle rank among Reeb graphs of Morse functions on $M$ is equal to the corank of $\pi_1(M)$, the fundamental group of $M$. Up to our knowledge, there are no tools to produce a~function whose Reeb graph has a given positive cycle rank. 

Having such a construction would also be of interest because of the use of Reeb graphs in topological data analysis, especially their algorithmic version, a mapper \cite{SinghMemoliCarlsson}. For example, an analysis of the first homology group of Reeb graphs, mappers and nerve complexes was performed by T. Dey, F. Memoli and Y.~Wang in \cite{DeyMemoliWang}. It is known that more accurate information about $H_1(\reeb(f)$) is encoded in $\pi_1(M)$ rather than in $H_1(M)$ since $\corank(\pi_1(M))$ can be arbitrary smaller than $\beta_1(M) = \rank_\Z(H_1(M))$.

The corank of a finitely generated group $G$ is the maximum rank of a free group onto which there is an epimorphism from $G$. By $\Delta_2(\reeb(f))$ we denote the number of vertices of degree $2$ in $\reeb(f)$. One of the main results of that paper is the following theorem, which shows how the problem is strongly related to the topology of the manifold. A crucial part of its proof is the calculation of a new invariant $\Omega_k(G)$ of a finitely presented group $G$ (see~Definition~\ref{definition:simplicity}).

\begin{theorem}
	\label{theorem:main_theorem}
	Let $M_{\pm 1}$ be the $\es^1$-bundle over an orientable surface $\Sigma_g$ of genus $g\geq 1$ with Euler number $e= \pm1$. Then $\corank(\pi_1(M_{\pm 1}))=g$ and
	$$
	\Omega_{2g}(\pi_1(M_{\pm 1})) = 2g.
	$$
	Consequently, the Reeb graph of any simple Morse function $f\colon M_{\pm 1} \to \R$ with 
	the minimum number of critical points is a tree, 
	i.e. $\beta_1(\reeb(f)) = 0$, and $\Delta_2(\reeb(f)) = 4g$.
	
	However, 
	if we increase the number of critical points by $2$, then there exists
	a simple Morse function $f'\colon M_{\pm 1} \to \R$ 
	such that
	$\beta_1(\reeb(f')) = g = \corank(\pi_1(M_{\pm 1}))$ and $\Delta_2(\reeb(f')) = 2g+2$.
\end{theorem}

The invariant $\Omega_k(G)$ measures the minimum number of generators among all finite presentations of $G$ of rank $k$ and the same deficiency as $G$, which are needed to write first relator and such that each next relator needs a one more generator and all the previous ones. Theorem \ref{theorem:simplicity_cycles_and_degree_2_vertices} shows its connection with simple Morse functions on $3$-manifolds with $k_i$ critical points of index $i$ and $k_0=1$:
$$
\Omega_{k_1}(\pi_1(M)) \leq k_1 - \beta_1(\reeb(f)).
$$
The calculation of $\Omega_{2g}(\pi_1(M_{\pm 1}))$ is done using an algebraic fundamental fact of one-relator groups, the \emph{Freiheitssatz} (see W. Magnus \cite{Magnus1930,Magnus-book}).

Another motivation for describing functions $f$ with cycle rank of $\reeb(f)$ equal to $\corank(\pi_1(M))$
is the fact that it is not so easy to compute the corank of a group in general. The only known general method are Makanin--Razborov diagrams \cite{Makanin_eqs_in_free_groups,Razborov}. This issue is related to the A. Tarski's problem of the existence of solution to a system of equation in a free finitely generated group, which solution was provided in the Z. Sela's works (\cite{Sela} and later). However, Makanin--Razborov diagrams are difficult to use in practice. J. Stallings \cite{Stallings_corank1} proposed to find an algorithmic method of computing the corank of $\pi_1(M)$ in terms of non-separating hypersurfaces in~$M$. In fact, $\corank(\pi_1(M))$ is equal to the maximum number of components in a non-separating two-sided hypersurface in~$M$. However, we are not aware of any computational method like that.

Thus we focus on the relation of corank with Reeb graphs of Morse functions. If $\beta_1(\reeb(f)) = \corank(\pi_1(M))$, then the function $f$ is far from being ordered. To form a cycle in $\reeb(f)$ it is necessary to have a critical point of index $n-1$ below a critical point of index $1$. The more disordered Morse function is, the greater the chance of having cycles in its Reeb graph. For simplicity, we restrict our considerations to simple Morse functions, having critical points with distinct critical values.

In the case of surfaces it is quite easy to construct a Morse function with a given cycle rank of its Reeb graph (cf. K. Cole-McLaughlin, H. Edelsbrunner et al. \cite{Edelsbrunner}), however, it depends on the number of vertices of degree $2$ (\cite[Theorem~5.6]{Michalak-TMNA}).
For higher-dimensional manifolds the problem is obviously much more complicated and again vertices of degree $2$ in Reeb graphs play an important role. Thus we pay our attention to the invariant $\Delta_2(M)$, which is the minimum number of vertices of degree $2$ among all Reeb graphs of simple Morse functions on $M$. 
There is a straightforward application of $\Delta_2(M)$ in the realization problem for Reeb graphs resolved in \cite{Marzantowicz-Michalak,Michalak-TMNA,Michalak-DCG} and O.~Saeki's paper~\cite{Saeki_Reeb_spaces}. Recall that any graph $\Gamma$ with the so-called good orientation and $\beta_1(\Gamma) \leq \corank(\pi_1(M))$ can be realized as the Reeb graph of a Morse function on $M$ up to orientation-preserving homeomorphism of graphs. 
A realization of $\Gamma$ up to isomorphism for a manifold of arbitrary dimension was provided by O. Saeki \cite{Saeki_Reeb_spaces}, but it uses smooth functions with infinitely many critical points, which even form submanifolds of codimension $0$. Such functions describe the topology of the manifold much more poorly. The vertices of degree $2$ in $\Gamma$ cause a restriction to its realization as the Reeb graph of a simple Morse function on $M$ --- it is impossible if $\Delta_2(\Gamma) < \Delta_2(M)$. 
As we will see $\Delta_2(M)$ can be described by topological invariants of $M$.

In this work it is shown that $\Delta_2(M)$ is non-zero in most cases. For example, by Proposition \ref{proposition:3-manifold_with_free_group} an orientable closed $3$-manifold $M$ has $\Delta_2(M)=0$ if and only if it is the connected sum of copies of $\es^2 \times \es^1$. Moreover, we prove that $\Delta_2(M)=2$ if and only if $M$ is the connected sum of copies of $\es^2 \times \es^1$ and a one copy of a lens space (Theorem \ref{theorem:Delta_2=2_iff_lens_space}).
Proposition \ref{proposition:inequalities_for_degree_2_vertices} provides three essentially different lower bounds on $\Delta_2(M)$ in terms of $\pi_1(M)$, homology groups of $M$ and Lusternik--Schnirelmann category of $M$. In the case of orientable $3$-manifolds, all of them can be improved by the inequality
\begin{equation}\label{formula:intro_inequality}
	\Delta_2(M) \geq 2(g(M) - \corank(\pi_1(M))),
\end{equation}
where $g(M)$ is the Heegaard genus of $M$. Thus we restrict our attention to orientable $3$-manifolds. 
It is reasonable to ask whether a Morse function with minimum number of critical points, so inducing a Heegaard splitting of minimum genus $g(M)$, can realize $\corank(\pi_1(M))$. Thus there are three natural questions.

\begin{problem}\label{problem:main}
	For an orientable $3$-manifold $M$ does there exists a simple Morse function $f\colon M \to \R$ with $k_i$ critical points of index $i$ such that 
	\begin{itemize}
		\item[(a)]  $k_1 = g(M)$ and $\beta_1(\reeb(f)) = \corank(\pi_1(M))$?
		\item[(b)] $k_1 = g(M)$ and $\Delta_2(\reeb(f)) = \Delta_2(M)$?
		\item[(c)]  $\beta_1(\reeb(f)) = \corank(\pi_1(M))$ and $\Delta_2(\reeb(f)) = \Delta_2(M)$?
	\end{itemize}
\end{problem}

The positive answer to the question (a) is equivalent to the equality in the bound (\ref{formula:intro_inequality}) (Corollary~\ref{corollary:equality_Delta_2(M)=2(g-corank)}). We show that it is not true for a manifold $M$ with $g(M) = \corank(\pi_1(M))+1$ which does not have a lens space as a summand in its prime decomposition (Corollary \ref{corollary:at_least_4_degree_2_vertices_if_genus_and_corank_differs_by_1}). The Heisenberg manifold, which is just a circle bundle $M_1$ for $g=1$, 
is a nice example of such a manifold. More generally, $M_1$ for $g\geq 2$ is 
a counterexample to the question (b). We do not knot whether (c) always holds.

The paper is organized as follows. Section \ref{section:basic_notions} provides needed preliminaries.
Next, in Section~\ref{section:vertices_of_degree_two}, we~present properties and bounds on the number of vertices of degree $2$ in Reeb graphs of simple Morse functions on a closed manifold of arbitrary dimension.
Section \ref{section:3-manifolds} deals with functions on $3$-manifolds. In Subsection \ref{subsection:low_values_of_Delta_2} we classify orientable \mbox{$3$-manifolds}~$M$ with $\Delta_2(M)=0$ or $2$. Subsection \ref{subsection:Omega_invariant} provides a definition and properties of $\Omega$-invariant of a group. Next, in Subsection \ref{subsection:circle_bundles}, considerations for the circle bundles are presented. Section \ref{section:final_remarks} states some further directions of research.


\section{Preliminaries}\label{section:basic_notions}

Throughout the paper $M$ is a closed smooth and connected manifold of dimension $n \geq 2$ and $\Sigma_g$ is a closed orientable surface of genus $g$.

Recall that a smooth function $f\colon M \to \R$ is a Morse function if all its critical points are non-degenerate, i.e. the Hessians at critical points are non-singular. Hereafter, we denote by $k_i$ the number of critical points of index $i$ of the considered function. A Morse function is \emph{simple} if each its critical level contains exactly one critical point. For the standard facts regarding Morse theory and its connections with handle decompositions we refer the reader to J. Milnor \cite{Milnor}.

The \emph{Reeb graph} $\reeb(f)$ of a Morse function $f$ is a finite graph obtained by contracting connected components of level sets of $f$ \cite{Reeb}. Its vertices correspond to connected components of level sets containing critical points. The Reeb graph can be defined also for more general classes of functions, e.g. for smooth functions with finitely many critical values (see O. Saeki \cite{Saeki_Reeb_spaces}).

The \emph{cycle rank} of a finite graph $\Gamma$ is defined to be its first Betti number $\beta_1(\Gamma)$. If $\Gamma$ is connected, then $\pi_1(\Gamma) \cong F_r$ is a free group $F_r$ of rank $r=\beta_1(\Gamma) = |E|-|V|+1$, where $V$ is the set of vertices and $E$ is the set of edges of $\Gamma$.

By \cite{KMS} of M. Kaluba et al. the quotient map $q_f \colon M \to \reeb(f)$ to the Reeb graph of $f$ induces an epimorphism $\pi_1(M) \to \pi_1(\reeb(f))$ onto a free group. We showed \cite{Michalak-DCG} (cf. O. Cornea \cite{Cornea}, I. Gelbukh \cite{Gelbukh:DCG} and W. Jaco \cite{Jaco})) the equivalence of the following three conditions:
\begin{itemize}
	\item There exists a Morse function (simple unless $M$ is an orientable surface) whose Reeb graph has cycle rank equal to $r$.
	\item There exists an epimorphism $\pi_1(M) \to F_r$.
	\item There exists a non-separating two-sided hypersurface (i.e. a codimension $1$ submanifold) in~$M$ with $r$ connected components.
\end{itemize}

The \emph{corank} of a finitely generated group $G$ is the maximum rank of a free group onto which there is an epimorphism from $G$. Thus the maximum cycle rank among all Reeb graphs of Morse functions on~$M$ is equal to $\corank(\pi_1(M))$. In fact, any integer between $0$ and $\corank(\pi_1(M))$ is a cycle rank of the Reeb graph of a simple Morse function on~$M$ different from an orientable surface.

Furthermore, there is a correspondence between epimorphisms $\pi_1(M) \to F_r$ and epimorphisms induced by the quotient maps $q_f$ on fundamental groups (see W. Marzantowicz, \L{}. Michalak \cite{Marzantowicz-Michalak} and O. Saeki \cite{Saeki_Reeb_spaces}).

By $\Delta(\Gamma)$ we denote the maximum degree of a vertex in a graph $\Gamma$, and by $\Delta_k(\Gamma)$ the number of its vertices of degree $k$.

For the definition of the \emph{canonical graph} and the Reeb graph in a \emph{canonical form} we refer to \cite[Definition~4.6]{Michalak-DCG}. Note that for a simple Morse function $f$ on a $3$-manifold the critical points of index $1$ and $2$ correspond bijectively with vertices of degree $2$ and $3$ in $\reeb(f)$ (\cite[Proposition 3.1]{Michalak-DCG}). Thus using the combinatorial modifications (1)--(3) from \cite[Lemma 4.1]{Michalak-DCG} we may assume that if $\reeb(f)$ is in a~canonical form, then there are no critical points of index $1$ (resp. index $2$) which are above (below) a critical point of index $2$ (resp. index $1$) corresponding to a vertex of degree $2$. In this case we say that $\reeb(f)$ is in \emph{the canonical form}. In brief, all vertices of degree $2$ corresponding to critical points of index $1$ (resp. index~$2$) are located just above (resp. below) the vertex of degree $1$ corresponding to the minimum (resp. maximum) of $f$.


\section{Bounds on the Number of Degree $2$ Vertices}\label{section:vertices_of_degree_two}

The results of \cite{Michalak-DCG, Marzantowicz-Michalak} concerned with the realization of graphs as Reeb graphs of Morse functions, hold up to homeomorphism for manifolds of dimension at least $3$. O. Saeki \cite{Saeki_Reeb_spaces} provided the realization up to isomorphism of graphs using function with infinitely many critical points. However, the number of vertices of degree $2$ in the Reeb graph of Morse function cannot be arbitrary. 
In the case of surfaces it is seen for example in \cite[Theorem 5.6]{Michalak-TMNA}.
In this section, we indicate basic properties of the number of degree $2$ vertices in Reeb graphs of simple Morse functions. It will play an important role in finding a Morse function on a given manifold whose Reeb graph has the maximum possible cycle rank.

\begin{definition}
	We define $\Delta_2(M)$ to be the minimum number of vertices of degree~$2$ in Reeb graphs of simple Morse functions on a closed manifold $M$.
\end{definition}

Clearly, $\Delta_2(M)$ is a diffeomorphism invariant of $M$. It is easy to check that $\Delta_2(\Sigma) = (\chi(\Sigma) \mod 2)$ for any closed connected surface $\Sigma$ (orientable or not), where $\chi(\Sigma)$ is its Euler characteristic.
Note the following fact which follows directly from the definition of $\Delta_2(M)$.

\begin{corollary}
	Let $\Gamma$ be a finite graph 
	and $\Delta(\Gamma) \leq 3$. If $\Delta_2(\Gamma) < \Delta_2(M)$, then there is no simple Morse function $f\colon M \to \R$ whose Reeb graph is isomorphic to $\Gamma$.
\end{corollary}

By \cite[Lemma 3.4]{Michalak-DCG} we have the following formula between the cycle rank of the Reeb graph of a simple Morse function $f\colon M \to \R$, numbers $k_i$ of critical points of index $i$ and the number of degree $3$ vertices:
\begin{align}\label{formula:cycle_rank_in_terms_of_Delta_3_Lemma_DCG}
	\beta_1(\reeb(f)) = -\frac{k_0+k_n}{2} + \frac{\Delta_3(\reeb(f))}{2} + 1.
\end{align}
In particular, if $f$ has only two extrema, then $k_0=k_n=1$ and so $\beta_1(\reeb(f)) = \frac{\Delta_3(\reeb(f))}{2}$.

\begin{proposition}\label{proposition:Euler_characteristic_is_Delta_2(M)_mod_2}
	Let $f\colon M \to \R$ be a simple Morse function on a closed manifold $M$. 
	Then
	$$
	\Delta_2(\reeb(f)) \equiv \chi(M) \mod 2.
	$$
\end{proposition}

\begin{proof}
	By \cite[Proposition 3.1]{Michalak-DCG} (cf. \cite[Theorem 3]{Reeb}) 
	$$
	\Delta_2(\reeb(f)) + \Delta_3(\reeb(f)) = k_1 + \ldots + k_{n-1},
	$$
	so using the above formula we obtain
	\begin{align*}
		\Delta_2(\reeb(f)) = k_1 + \ldots + k_{n-1} - k_0 - k_n - 2\beta_1(\reeb(f))+2  
		\equiv  \sum_{i=0}^n (-1)^i k_i = \chi(M) \ \ \ ({\rm mod}\ 2).
	\end{align*}
\end{proof}

The \emph{rank} of a finitely generated group $G$ is the smallest cardinality of its generating set. It is clear that $\rank(G) \geq \corank(G)$. Note that $k_1 \geq \rank \pi_1(M)$, since a Morse function $f$ leads to a CW-decomposition with $k_1$ cells of dimension~$1$. The same argument for $-f$ gives us $k_{n-1} \geq \rank \pi_1(M)$.

By $\cat(X)$ we denote the \emph{Lusternik–Schnirelmann category} of a space~$X$, the minimum number of open sets covering $X$ which are contractible in $X$. Thus $\cat(X)=1$ if $X$ is contractible.

The following proposition gives bounds on $\Delta_2(M)$ in terms of $\pi_1(M)$, $\cat(M)$ and homology groups $H_*(M,R)$ with coefficients in a principal ideal domain $R$. Note that $\rank_R H_i(M,R)$ is the rank over $R$ of a finitely-generated $R$-module.

\begin{proposition}\label{proposition:inequalities_for_degree_2_vertices}
	Let $M$ be a closed manifold of dimension $n\geq 3$. Then $\Delta_2(M)$ satisfies the inequalities
	\begin{align}\label{formula:Delta_2>= 2(rank-corank)}
		\Delta_2(M) \geq 2(\rank(\pi_1(M)) - \corank(\pi_1(M))),
	\end{align}
	\begin{align}\label{formula:Delta_2>= sum_dim_H_i-2corank)}
		\Delta_2(M) \geq \sum_{i=1}^{n-1} \rank_R H_i(M,R)-2\corank(\pi_1(M)),
	\end{align}
	\begin{align}\label{formula:Delta_2>= cat(M)-2corank-2)}
		\Delta_2(M) \geq \cat(M)-2\corank(\pi_1(M)) -2.
	\end{align}
	Moreover, if $M$ is an orientable $3$-manifold with Heegaard genus $g(M)$, then
	\begin{align}\label{formula:Delta_2>= 2(g(M)-corank)}
		\Delta_2(M) \geq 2(g(M) - \corank(\pi_1(M))).
	\end{align}
	
\end{proposition}

Basic information about Heegaard splittings and other facts of $3$-manifold topology can be found in J.~Hempel's book \cite{Hempel}.

\begin{proof}
	Let $f\colon M \to\R$ be a simple Morse function. By \cite[Lemma 4.5]{Michalak-DCG} we may assume that $f$ has only two extrema without changing $\Delta_2(\reeb(f))$. Therefore $k_0=k_n=1$ and 
	$\corank(\pi_1(M)) \geq \beta_1(\reeb(f)) = \frac{\Delta_3(\reeb(f))}{2}$, so
	\begin{align*}
		\Delta_2(\reeb(f))  = k_1 + \ldots + k_{n-1} - \Delta_3(\reeb(f))  &\geq 2\left( \frac{k_1+k_{n-1}}{2} - \frac{\Delta_3(\reeb(f)) }{2}\right) 	\\
		&\geq 2\left(\rank(\pi_1(M)) - \corank(\pi_1(M))\right).	
	\end{align*}
	From Morse inequalities $k_i \geq \rank_R H_i(M,R)$ we obtain (\ref{formula:Delta_2>= sum_dim_H_i-2corank)}). It is also known (see \cite{Takens}) that $\cat(M)$ bounds from below the number of critical points of a smooth function on $M$. Thus $\sum_{i=0}^n k_i \geq \cat(M)$, what gives (\ref{formula:Delta_2>= cat(M)-2corank-2)}).
	
	Finally, if $M$ is an orientable $3$-manifold, then $k_1 = k_2 \geq g(M)$ since $f$ has exactly two extrema and $\chi(M)=0$. Therefore
	$$
	\Delta_2(\reeb(f))  = k_1 +  k_{2} - \Delta_3(\reeb(f))  = 2\left( \frac{k_1+k_2}{2} - \frac{\Delta_3(\reeb(f)) }{2}\right) = 2\left( k_1 - \beta_1(\reeb(f))\right),
	$$
	what implies the desired inequality.
\end{proof}

\begin{example}
	Since $\cat(M) \leq \dim M +1$ and $\cat(M)$ is not an easy invariant to compute, the inequality (\ref{formula:Delta_2>= cat(M)-2corank-2)}) seems to be of less utility. However, for $M=\R\!\operatorname{P}^n$ the bound (\ref{formula:Delta_2>= cat(M)-2corank-2)}) provides $\Delta_2(\R\!\operatorname{P}^n) \geq n-1$ since $\cat(\R\!\operatorname{P}^n)=n+1$, while (\ref{formula:Delta_2>= 2(rank-corank)}) gives only $\Delta_2(\R\!\operatorname{P}^n)\geq 2$.
	
	More generally, if $M$ is simply connected, than (\ref{formula:Delta_2>= 2(rank-corank)}) is trivial, but (\ref{formula:Delta_2>= cat(M)-2corank-2)})  yields $\Delta_2(M)\geq \cat(M)-2$ and the right-hand side is positive if $M$ is not a sphere. For example, $\cat(\C\!\operatorname{P}^n)=n+1$, so $\Delta_2(\C\!\operatorname{P}^n)\geq n-1$.
	
	In both the examples bounds given by (\ref{formula:Delta_2>= sum_dim_H_i-2corank)}) and (\ref{formula:Delta_2>= cat(M)-2corank-2)}) are the same. It may also happen that (\ref{formula:Delta_2>= cat(M)-2corank-2)}) is better than (\ref{formula:Delta_2>= sum_dim_H_i-2corank)}). For example, if $M$ is a homology sphere other than $\es^n$, then (\ref{formula:Delta_2>= sum_dim_H_i-2corank)}) yields $\Delta_2(M) \geq 0$, but (\ref{formula:Delta_2>= cat(M)-2corank-2)}) gives $\Delta_2(M) \geq \cat(M)-2 \geq 2$, since $\cat(M) \geq 4$ if $\pi_1(M)$ is not free by \cite{LS-cat}.
\end{example}

\begin{example}
	In some cases the bound (\ref{formula:Delta_2>= 2(rank-corank)}) can be better than (\ref{formula:Delta_2>= sum_dim_H_i-2corank)}) and (\ref{formula:Delta_2>= cat(M)-2corank-2)}).
	
	Take $M= L_p \# L_q$, the connected sum of two lens spaces such that $\gcd(p,q)=1$, where $\pi_1(L_k)=\Z/k\Z$. Then $\pi_1(M) = (\Z/p\Z )* (\Z/q\Z)$, $\rank(\pi_1(M))=2$, $\corank(\pi_1(M))=0$, so (\ref{formula:Delta_2>= 2(rank-corank)}) gives $\Delta_2(M) \geq 4$. Since $\cat(M)\leq 4$, the bound (\ref{formula:Delta_2>= cat(M)-2corank-2)}) yields at most $\Delta_2(M)\geq 2$. Moreover, $H_0(M)=H_3(M)=\Z$, $H_1(M)=\Z/pq\Z$ and $H_2(M)=0$, so $\rank_R H_1(M,R) \leq 1$ and $\rank_R H_2(M,R) \leq 1$ for any principal ideal domain $R$ by the universal coefficient theorem. Thus from (\ref{formula:Delta_2>= sum_dim_H_i-2corank)}) we also obtain at most $\Delta_2(M) \geq 2$
	
\end{example}

\begin{example}\label{example:n-torus_and_Delta_2}
	Similarly, there are examples where the bound  (\ref{formula:Delta_2>= sum_dim_H_i-2corank)}) is better than (\ref{formula:Delta_2>= 2(rank-corank)}) and (\ref{formula:Delta_2>= cat(M)-2corank-2)}).
	
	For $n$-dimensional torus $T^n$ one can show that $\cat(T^n)=n+1$, so the inequality (\ref{formula:Delta_2>= cat(M)-2corank-2)}) implies $\Delta_2(T^n) \geq n-3$, while (\ref{formula:Delta_2>= 2(rank-corank)}) gives $\Delta_2(T^n) \geq 2(n-1)$. However, since $\rank_\Z H_k(T^n) = \binom{n}{k}$, the formula (\ref{formula:Delta_2>= sum_dim_H_i-2corank)}) provides $\Delta_2(M) \geq 2^n - 4$.
\end{example}

\begin{example}
	In the case of orientable $3$-manifolds the bound (\ref{formula:Delta_2>= 2(g(M)-corank)}) using Heegaard genus $g(M)$ is sharper than the other three, because $g(M)\geq \rank(\pi_1(M))$, $g(M) \geq \rank_R H_i(M,R)$ for $i=1,2$ and $2g(M)\geq \cat(M)-2$ (the last inequality follows since $\cat(M) -2\leq 2$, so it suffices to $g(M)\geq 1$, and for $g(M)=0$, $M$ is the $3$-sphere, so $\cat(M)=2$).
	
	By \cite[Proposition 3.2]{Michalak-DCG} the Reeb graph of an ordered simple Morse function with $k_1 = g(M)$ is a tree, so it has $2g(M)$ vertices of degree $2$. Therefore 
	\begin{align}\label{formula:upper_and_lower_bounds_for_delta_2_using_Heegaard_genus}
		2g(M) \geq \Delta_2(M) \geq 2(g(M) - \corank(\pi_1(M))).
	\end{align}
	In particular, if $\corank(\pi_1(M))=0$, e.g. if $M$ is a homology sphere, then $\Delta_2(M) = 2g(M)$.
\end{example}

\begin{lemma}\label{lemma:gluing_functions_for_connected_sum}
	\label{lemma:additivity_for_connected_sum_when_equality_in_lower_bounds}
	For closed manifolds $M_1$ and $M_2$ of the same dimension we have
	$$
	\Delta_2(M_1 \# M_2) \leq \Delta_2(M_1) + \Delta_2(M_2).
	$$	
	Moreover, if both $M_1$ and $M_2$ have equality in one of (\ref{formula:Delta_2>= 2(rank-corank)}), (\ref{formula:Delta_2>= 2(g(M)-corank)}) or in (\ref{formula:Delta_2>= sum_dim_H_i-2corank)}) for $R=\Z$ if one of $M_i$ is orientable, then 
	$
	\Delta_2(M_1 \# M_2) = \Delta_2(M_1) + \Delta_2(M_2).
	$
\end{lemma}

\begin{proof}
	For $i=1,2$ let $f_i\colon M_i \to \R$ be a simple Morse function on $M_i$ such that $\Delta_2(\reeb(f_i)) = \Delta_2(M_i)$. After possible translation we may assume that the maximum value of $f_1$ is the minimum value of $f_2$, which is denoted by $c$. If $f_i(p_i)=c$, take an $n$-handle $D^n_i$ corresponding to $p_1$ and a $0$-handle corresponding to $p_2$, so the boundary $\partial D^n_i$ is the component of level set of $f_i$. Perform the connected sum $M_1 \# M_2$ removing $\int D^n_i$. Then the functions $f_i$ paste together to a simple Morse function $f$ on $M_1 \# M_2$ such that $\Delta_2(\reeb(f)) = \Delta_2(M_1) + \Delta_2(M_2)$.
	
	The second statement is clear since $\rank(\pi_1(M))$ (by Grushko Theorem \cite{Magnus-book}), $\corank(\pi_1(M))$ (by \cite{Cornea}), $g(M)$ (by Haken Theorem \cite{Haken}) and $H_i(M)$ for $i=1, \ldots,n-1$ are additive with respect to the connected sum operation (in the case of $H_{n-1}(M_1 \# M_2)$ one of $M_i$ must be orientable).
\end{proof}

\begin{remark}
	Note that for a given $n\geq 3$ the number $\Delta_2(M)$ can be arbitrary large among all smooth $n$-manifolds $M$ without using the connected sum operation. Simply, take $M = \Sigma_g \times \es^{n-2}$. Then $\rank(\pi_1(M)) \geq 2g$ and $\corank(\pi_1(M)) = g$, so $\Delta_2(M) \geq 2g$ by the inequality (\ref{formula:Delta_2>= 2(rank-corank)}).
\end{remark}


\section{Simple Morse Functions on $3$-Manifolds}\label{section:3-manifolds}

Hereafter, $M$ is a closed orientable $3$-manifold. Note that, by the proof of Proposition~\ref{proposition:inequalities_for_degree_2_vertices}, $\Delta_2(\reeb(f)) = 2(k_1 - \beta_1(\reeb(f))$ for a simple Morse function $f\colon M \to \R$ with exactly two extrema. The equality in (\ref{formula:Delta_2>= 2(g(M)-corank)}) gives a positive answer to all questions in Problem \ref{problem:main}.

\begin{corollary}\label{corollary:equality_Delta_2(M)=2(g-corank)}
	The following are equivalent:
	\begin{itemize}
		\item[(a)] There is a simple Morse function $f\colon M \to \R$ such that $k_1 = g(M)$ and $\beta_1(\reeb(f)) = \corank(\pi_1(M))$.
		\item[(b)] $\Delta_2(M) = 2 (g(M)-\corank(\pi_1(M)))$.
		\item[(c)] Any simple Morse function $f\colon M \to \R$ with exactly two extrema and $\Delta_2(\reeb(f)) = \Delta_2(M)$ has $k_1 = g(M)$ and $\beta_1(\reeb(f)) = \corank(\pi_1(M))$.
	\end{itemize}
\end{corollary}

\begin{proof}
	By the formula $\Delta_2(\reeb(f)) = 2(k_1 - \beta_1(\reeb(f)))$ the only non-trivial implication is from $(1)$ or $(2)$ to $(3)$. If $(2)$ holds and $f\colon M \to \R$ is a simple Morse function such that $k_0 =1 = k_3$ and $\Delta_2(\reeb(f)) = \Delta_2(M)$, then $2(g(M) - \corank(\pi_1(M))) = 2(k_1 - \beta_1(\reeb(f)))$. Thus 
	$$
	0 = (k_1 - g(M)) + (\corank(\pi_1(M)) - \beta_1(\reeb(f))) 
	$$
	and both the expressions in brackets are non-negative. Therefore $k_1 = g(M)$ and $\beta_1(\reeb(f))=\corank(\pi_1(M))$.
\end{proof}

\subsection{Low values of $\Delta_2(M)$}\label{subsection:low_values_of_Delta_2}

Now, we will relate the condition $\Delta_2(M)=0$ to the known facts of $3$-dimensional topology and express them in terms of $\corank(\pi_1(M))$.

\begin{proposition}\label{proposition:3-manifold_with_free_group}
	Let $M$ be an orientable closed $3$-manifold and $r=\corank(\pi_1(M))$. The following are equivalent:
	\begin{itemize}
		\item[(a)] $M \cong \#_{i=1}^r \left( \es^2 \times \es^1 \right)$ (for $r=0$ it means $M \cong \es^3$);
		\item[(b)] $\Delta_2(M)=0$;
		\item[(c)] $g(M) = r$;
		\item[(d)] $\rank(\pi_1(M)) = r$;
		\item[(e)] $\pi_1(M) \cong F_{r}$.
	\end{itemize}
\end{proposition}

\begin{proof}
	The manifold $\es^2\times \es^1$ has a Heegaard splitting of genus $1$ such that a $2$-handle is attached along an embedding $\varphi \colon \es^1 \times \de^1 \to \partial H$, where $\varphi({\es^1 \times \{0\}})$ is a meridian $\partial \de^2 \times \{1\}$ of the solid torus $H=\de^2 \times \es^1$. Thus we may assume that the image of $\varphi$ is disjoint from a $1$-handle attached to $\de^3$ to form $H$, so these handles can be attached in any order. Therefore a Morse function $f$, corresponding to a handle decomposition with $2$-handle attached before a $1$-handle (see \cite[Theorem 3.12]{Milnor}), has a unique critical point of index $2$ below a unique critical point of index~$1$. Hence $f$ is simple and since the $2$-handle is attached to the boundary of the ball $\de^3$, it splits the level set of $f$ into two parts. Therefore $\reeb(f)$ has no vertices of degree $2$ and  $\beta_1(\reeb(f))=1$. Thus (a) implies (b) by Lemma \ref{lemma:gluing_functions_for_connected_sum}.
	
	By (\ref{formula:upper_and_lower_bounds_for_delta_2_using_Heegaard_genus}) we get (c) from (b).	The implication (c) to (d) follows by $\corank(\pi_1(M)) \leq \rank(\pi_1(M) \leq g(M)$.

	Now, assume (d) and take an epimorphism $\varphi \colon \pi_1(M) \to F_r$. If $\{a_1,\ldots,a_r\}$ generates $\pi_1(M)$, then $S=\{\varphi(a_1),\ldots,\varphi(a_r)\}$ generates $F_r$. This generating set is equivalent, under Nielsen transformations, to free generating set of $F_r$ with also $r$ elements, so $S$ is also a free generating set. Thus a homomorphism $\phi\colon F_r \to \pi_1(M)$ defined on $S$ by $\phi(\varphi(a_i))=a_i$ is an inverse for $\varphi$, so they are isomorphisms and (e) holds.
	
	Finally, the implication (e) to (a) is a standard fact of $3$-manifold topology together with the Poincar{\'e}  conjecture to eliminate homotopy spheres in the prime decomposition of $M$ (see~e.g.~\cite{Hempel}).
\end{proof}

\begin{theorem}\label{theorem:Delta_2=2_iff_lens_space}
	For an orientable $3$-manifold $M$, $\Delta_2(M)=2$ if and only if $M= \left(\mathop{\Huge \#}_{i=1}^r \left( \es^2 \times \es^1 \right) \right) \# L $, where $r=\corank(\pi_1(M))$ and $L$ is a lens space.
\end{theorem}

\begin{proof}
	
	If $M= \left(\mathop{\Huge \#}_{i=1}^r \left( \es^2 \times \es^1 \right) \right) \# L $, then $2 \leq \Delta_2(M) \leq \Delta_2(L)=2$ by Lemma \ref{lemma:gluing_functions_for_connected_sum}, the inequality (\ref{formula:Delta_2>= 2(g(M)-corank)}) and any example of simple Morse function on $L$ inducing a Heegaard splitting of genus $g(L)=1$. Thus $\Delta_2(M)=2$.
	
	Conversely, assume that $\Delta_2(M)=2$. The inequality (\ref{formula:Delta_2>= 2(g(M)-corank)}) implies that $1\geq g(M)-\corank(\pi_1(M))$, hence either $g(M)=\corank(\pi_1(M))$ and then $\Delta_2(M)=0$ by Proposition \ref{proposition:3-manifold_with_free_group}, a contradiction, or $g(M)=\corank(\pi_1(M))+~1$. 
	Let $f\colon M \to \R$ be a simple Morse function with $\reeb(f)$ in the canonical form and $\Delta_2(\reeb(f)) = 2$, so $k_1 = g(M)$ and $\beta_1(\reeb(f)) = \corank(\pi_1(M)) = g(M)-1$ by Corollary \ref{corollary:equality_Delta_2(M)=2(g-corank)}. Starting from the minimum of $f$ and looking at successive critical points, they form the following list of indices: $0,1,2,1,2,1,\ldots,2,1,2,3$. First $\beta_1(\reeb(f))$ critical points of index $2$ correspond to handles attached to a manifold with boundary $T^2$ that split a level set into two components $T^2$ and $\es^2$, and the last $\beta_1(\reeb(f))$ critical points of index $1$ correspond to handles which merge these two components of level sets to a one being $T^2$.

	Therefore we can find a non-separating two-sided hypersurface $\N$ in $M$ with $\beta_1(\reeb(f)) = g(M)-1$ connected components all being $S^2$. Let $M|\N$ be the manifold obtained by cutting $M$ along $\N$. Re-gluing boundary components leads to a description of $\pi_1(M)$ as an HNN extension of $\pi_1(M |\N)$ with $g(M)-1$ stable letters. Since glued submanifolds are simply-connected, this extension is trivial and so $\pi_1(M) \cong \pi_1(M |\N) * F_{g(M) -1}$. Therefore $M = \left(\mathop{\Huge \#}_{i=1}^{g(M)-1} \left( \es^2 \times \es^1 \right) \right) \# M'$ by Knesser conjecture (see \cite{Hempel}) and the previous proposition. Thus $M'$ is a lens space since its Heegaard genus is $g(M')=1$ and $\corank(\pi_1(M')) = 0$ by the additivity of these numbers with respect to the connected sum operation (cf. \cite{Haken} and \cite{Cornea}, respectively).	
\end{proof}

As a conclusion, we have found a sufficient condition on $M$ to be a counterexample to the question (a) in Problem \ref{problem:main}.

\begin{corollary}\label{corollary:at_least_4_degree_2_vertices_if_genus_and_corank_differs_by_1}
	If $M$  does not have a lens space 
	in its prime decomposition and $g(M)=\corank(\pi_1(M))+1$, then $\Delta_2(M) \geq 4 > 2(g(M)-\corank(\pi_1(M)))$ and so any simple Morse function $f\colon M \to \R$ with $k_1 = g(M)$ has the strict inequality $\beta_1(\reeb(f)) < \corank(\pi_1(M))$. 
\end{corollary}

\subsection{Group presentation invariant}\label{subsection:Omega_invariant}

Consider a group presentation $\P= \geng{ x_1,\ldots, x_n\, |\, r_1, \ldots, r_m }$ with $\rank(\P) = n$ generators and $m$ relators $r_i = r_i(x_1,\ldots,x_n)$. Its \emph{deficiency} ${\rm def}(\P)$ is defined as $n-m$. The \emph{deficiency} ${\rm def}(G)$ of a finitely presented group $G$ is the maximal deficiency among all its finite presentations.

Note that a genus $g$ Heegaard spliting of a manifold $M$ induce a presentation of $\pi_1(M)$ with deficiency $0$. 
D. Epstein \cite{Epstein} showed that ${\rm def}(\pi_1(M)) = 0$ if $M$ is a closed orientable $3$-manifold.

We present the construction of an invariant of $\pi_1(M)$ in terms of its presentations, which will be essential in the study of Problem \ref{problem:main}.

\begin{definition}\label{definition:simplicity}
	Let $\P=\geng{ x_1,\ldots, x_n \,|\, r_1, \ldots, r_k }$ be a presentation of a non-free group, so $k\geq 1$. 
	We define $\Omega = \Omega(\P)$ to be the minimum positive number such that 
	for $1 \leq i \leq \min \{n-\Omega, k\}$ the relator $r_i$ can be written as a word in only first $\Omega+i-1$ generators, i.e. $r_i = r_i(x_1,\ldots,x_{\Omega+i-1})$.
	
	Thus $r_1$ uses first $\Omega$ generators, $r_2$ uses first $\Omega +1$ generators and so on. 
	Clearly, $\Omega(\P) \leq n = \rank(\P)$.
	
	For a finitely presented non-free group $G$ and a natural number $n$ we define the number $\Omega_n(G)$ to be
	$$
	\Omega_{n}(G) := \min \left\{\Omega(\P) \,\colon\, G \cong \P,  \operatorname{def}(\P)= \operatorname{def}(G) \text{ and } \rank(\P)=n \right\},
	$$
	if there exists such a presentation $\P$ of $G$, or $\Omega_n(G) = \infty$ otherwise.
	Note that $\Omega(\P)\neq 0$ since presentations of maximal deficiency have no trivial relators. Moreover,  $\Omega_n(G) \geq \Omega_{n+1}(G)$, which follows by adding a new generator $y$ and the relator $y$ to a presentation realizing $\Omega_n(G)$ if it exists. Thus the condition $\rank(\P)=n$ in the definition of $\Omega_n(G)$ can be equivalently substituted by $\rank(\P) \leq n$. Hence
	$$
	n \geq \Omega_n(G) \geq \Omega_{n+1}(G) \geq \ldots \geq 1,
	$$	
	if $G$ admits a presentation of deficiency $\operatorname{def}(G)$ with $n$ generators. Therefore, this sequence stabilizes for some number of generators and we define the number $\Omega(G)$ to be the minimum over all $\Omega_n(G)$, i.e.
	$$
	\Omega(G) := \min_n \Omega_n(G) = \min \{\Omega(\P) \,\colon\, G \cong \P,  \operatorname{def}(\P)=\operatorname{def}(G) \}.
	$$
\end{definition}

We do not know whether any finitely presented group $G$ has a presentation of deficiency $\operatorname{def}(G)$ and $\rank(G)$ generators. The paper \cite{Rapaport} of E. Rapaport provides some results on this interesting problem.

For a $3$-manifold $M$ with non-free $\pi_1(M)$ there are bounds 
$$
g(M) \geq \Omega_{g(M)}(\pi_1(M)) \geq \Omega(\pi_1(M)) \geq 1.
$$
Note that $\rank(\pi_1(M))$ can be smaller than $g(M)$ (see examples of such Seifert manifolds in \cite{Zieschang} of  M. Boileau and H. Zieschang).

\begin{lemma}\label{lemma:simplicity_of_torsion_free_group}
	If $G$ is a non-trivial, non-free, torsion-free group, then $\Omega(G) \geq 2$.
\end{lemma}

\begin{proof}
	Let $\P=\geng{ x_1,\ldots, x_n\, |\, r_1, \ldots, r_k }$ be a presentation of $G$ realizing $\Omega(G)$.
	If $\Omega(G)=1$, then $r_1 = x_1^m$, where $m$ is non-zero since $\operatorname{def}(G)=\operatorname{def}(\P)$. If $m=\pm 1$, then $x_1 = 1$ in $G$ and so deleting $x_1$ and $r_1$ and all occurrences of $x_1$ in other relators (a Tietze transformation) we obtain another presentation of $G$ which realizes $\Omega(G)$. Thus we may assume that $|m| \geq 2$. However, if still $x_1=1$ in $G$, so $x_1$ is a consequence of relators $r_1,\ldots,r_k$, then we replace the relator $r_1=x_1^m$ by $x_1$ (using Tietze transformations we add $x_1$ as a new relator and then remove $r_1$ as a consequence of $x_1$). Again, delete $r_1$ and $x_1$. Since this operation decreases rank of the presentation, we may finally assume that $r_1 = x_1^m$, $|m| \geq 2$ and $x_1 \neq 1$ in $G$. Thus $x_1$ is of finite order in $G$, a contradiction.	
\end{proof}

\begin{example}
	The discrete Heisenberg group $H(3,\Z)$ is generated by two elements and is torsion-free (it can be see algebraically or by noting that the Heisenberg manifold is aspherical), so $\Omega(H(3,\Z)) = 2$.
\end{example}

\begin{theorem}\label{theorem:simplicity_cycles_and_degree_2_vertices}
	Let $f\colon M \to \R$ be a simple Morse function on an orientable closed $3$-manifold $M$ with $k_i$ critical points of index $i$ and $k_0=1$. 
	Then
	$$
	\Omega_{k_1}(\pi_1(M)) \leq k_1-\beta_1(\reeb(f)) = \frac{1}{2}\Delta_2(\reeb(f)).
	$$
	As a consequence, 
	$$
	2(\Omega(\pi_1(M))) \leq \Delta_2(M).
	$$
\end{theorem}

\begin{proof}
	We may assume that $\reeb(f)$ is in the canonical form without changing $k_1$, $\beta_1(\reeb(f))$ and $\Delta_2(\reeb(f))$ (\cite[Proposition 4.7]{Michalak-DCG}). 
	The function $f$ leads to a Heegaard splitting of $M$ of genus $k_1$, and so to a presentation $\P$ of $\pi_1(M)$ with $k_1$ generators, corresponding to critical points $p_1,\ldots,p_{k_1}$ of index~$1$, and $k_1$ relators, corresponding to critical points $q_1, \ldots, q_{k_1}$  of index $2$. 
	
	Let $r=\beta_1(\reeb(f))$. Since $\reeb(f)$ is in the canonical form, we have the following sequence of critical values of~$f$:
	\begin{align*}
		f(p_1) &< \ldots < f(p_{k_1-r}) =: c \\
		c < f(q_1) < f(p_{k_1-r+1}) &< f(q_2) <  f(p_{k_1-r+2}) < \ldots < f(q_r) < f(p_{k_1}) =: c' \\
		c'  &< f(q_{r+1}) < f(q_{r+2}) < \ldots f(q_{k_1}).
	\end{align*}
	
	Therefore for each $1\leq i \leq r$ the $2$-handle corresponding to $q_i$ is attached to the boundary of the manifold consisting of  $k_1-r+i-1$ first $1$-handles, so the corresponding relator $r_i$ of $\P$ can be written as a word in only first $k_1-r+i-1$ generators. Thus $\Omega_{k_1}(\pi_1(M)) \leq k_1-r$.
	
	As in the proof of Proposition \ref{proposition:Euler_characteristic_is_Delta_2(M)_mod_2}, $\Delta_2(\reeb(f)) = 2(k_1 - r)$, what gives the desired characterization.
\end{proof}

As an easy application of the above theorem, if $M$ is a closed orientable $3$-manifold such that $\pi_1(M)$ is torsion-free and $g(M)=\corank(\pi_1(M))+1 \geq 2$ (e.g. $M$ is irreducible), then $\Omega(\pi_1(M)) \geq 2$ by Lemma \ref{lemma:simplicity_of_torsion_free_group} and so $\Delta_2(M) \geq 4$, what we proved earlier (Corollary \ref{corollary:at_least_4_degree_2_vertices_if_genus_and_corank_differs_by_1}).

\subsection{Circle bundles over an orientable surface}\label{subsection:circle_bundles}

A nice class of $3$-manifolds are orientable $\es^1$-bundles over a closed orientable surface $\Sigma_g$. They are classified by elements of $H^2(\Sigma_g,\Z) = \Z$, so any $e\in \Z$ corresponds to a bundle $M_e$ and conversely, any circle bundle over $\Sigma_g$ is isomorphic to $M_e$ for some $e\in \Z$. The number $e$ is the Euler number of the bundle $M_e \to \Sigma_g$. Note that $M_0 = \Sigma_g \times \es^1$ is the trivial bundle.

If $g=0$, then $M_e$ is a lens space $L(e,1)$ with the exception of $L(0,1) = \es^2 \times \es^1$ and $L(\pm 1,1) = \es^3$. Therefore this case if covered by Proposition \ref{proposition:3-manifold_with_free_group} and Theorem~\ref{theorem:Delta_2=2_iff_lens_space}.

Now, assume $g\geq 1$. The fundamental group of $M_e$ is an extension of $\pi_1(\Sigma_g)$ by $\pi_1(\es^1) \cong \Z$ and has the following presentation
$$
\pi_1(M_e) = \geng{a_1,b_1,\ldots,a_g,b_g,h \,|\, [a_i,h]=[b_i,h]=1 \text{ for $i=1,\ldots,g$ }, \prod_{i=1}^g [a_i,b_i] = h^e }.
$$	
Since
$$
H_1(M_e) = \ab(\pi_1(M_r)) = \Z^{2g}\times \Z/{e\Z},
$$
we have $\rank(\pi_1(M_e)) = 2g+1$ for $e\neq \pm 1$ and $\rank(\pi_1(M_e)) = 2g$ for $e=\pm1$. Furthermore, if an epimorphism $\varphi\colon \pi_1(M_e) \to F_r$ does not factorize through $\pi_1(M_e) \to \pi_1(\Sigma_g)$ induced by the bundle map, then $\varphi(h) \neq 1$ and so $r=1$ since $h$ is in the center of $\pi_1(M_e)$. Thus $\corank(\pi_1(M_e))=\corank(\pi_1(\Sigma_g)) = g$. 

Moreover, it is known that $g(M_e) = 2g+1$ for $e \neq \pm 1$ (Figure \ref{figure:heegaard_splitting_of_M_e} shows a Heegaard splitting of $M_e$ of such genus). M. Boileau and H. Zieschang \cite{Zieschang} showed that $g(M_{\pm 1}) = 2g$.

\begin{proposition}\label{proposition:Heegaard_splitting_of_M_e}
	There exists a simple Morse function $f\colon M_e \to \R$ such that $\beta_1(\reeb(f)) = g$, $\Delta_2(\reeb(f)) = 2g+2$ and which induces a Heegaard splitting of $M_e$ of genus $2g+1$.
	
	Thus for $e \neq \pm 1$ equality holds in (\ref{formula:Delta_2>= 2(g(M)-corank)}):
	$$
	\Delta_2(M_e) = 2(g(M_e) - \corank(\pi_1(M_e))) = 2g+2.
	$$
\end{proposition}

\begin{proof}
	Figure \ref{figure:heegaard_splitting_of_M_e} presents a Heegaard diagram of $M_e$. It is a sphere with $2g+1$ handles formed by identifying circles: $H^+$ with $H^-$, $A_i^+$ with $A_i^-$ and $B^+_i$ with $B^-_i$ for $i=1,\ldots,g$, preserving the orientations indicated by arrows in the figure. The circles $H$, $A_i$, $B_i$ obtained after gluing constitute a system of curves defining a handlebody $\mathcal{H}$, and the curves $\alpha_i$, $\beta_i$, $i=1,\ldots,g$ and $\gamma$ define the second handlebody in the considered Heegaard splitting. Let $M$ be the resulting closed $3$-manifold and $\Sigma \cong \Sigma_{2g+1}$ be a Heegaard surface of this splitting. Choose generators $a_1,b_1,\ldots,a_g,b_g,h$ of $\pi_1(\H) = F_{2g+1}$ associated with this diagram, i.e.~a~loop in $\Sigma = \partial \mathcal{H}$ representing $h$ intersects $H$ once and is disjoint from $A_i, B_i$, $i=1,\ldots,g$,
	and similarly for $i=1,\ldots,g$ a loop in $\Sigma$ representing $a_i$ ($b_i$~respectively) intersects only $A_i$ (resp. $B_i$) and only once. Using these generators, the curve $\alpha_i$ represents the word $[a_i,h]$, $\beta_i$ represents $[b_i,h]$, and $\gamma$ represents $h^{-e}\cdot \prod_{i=1}^g [a_i,b_i]$. Note that for $e>0$ we need to change $\gamma$ in the obvious way, hitting $H^-$ instead of $H^+$ as in the figure. Therefore $\pi_1(M)$ has the same presentation as $\pi_1(M_e)$ written above, and so by the Waldhausen Theorem \cite[Corollary 6.5]{Waldhausen} $M$ is diffeomorphic to $M_e$.
	
	\begin{figure}[h]
		\centering

		\begin{tikzpicture}[scale=1]
			\node[rotate=180] at (0,0) {\includegraphics[width=125mm, height=100mm]{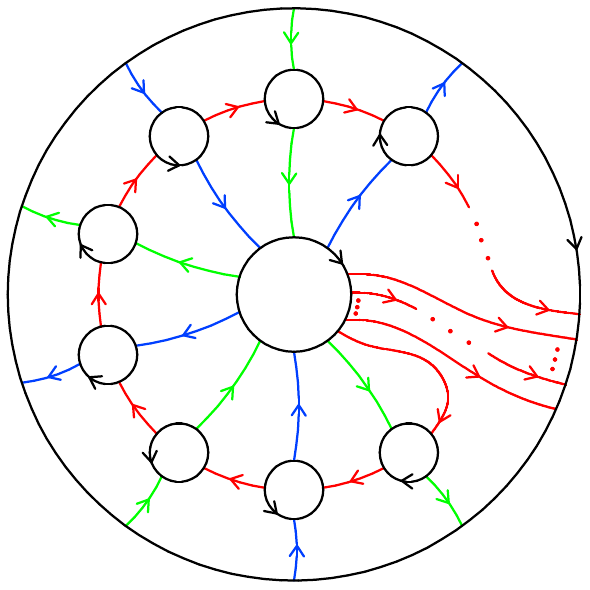}};

			\def\rrr{3.3}
			\def\RRR{4.125}
			
			\draw (-4,2) node {\Large \textcolor{red}{\textbf{$\gamma$}}};

			\draw ({1.3*\RRR*cos(120)},{1.3*\rrr*sin(120)}) node {\Large \textcolor{green}{\textbf{$\alpha_1$}}};
			\draw ({1.3*\RRR*cos(48)},{1.3*\rrr*sin(48)}) node {\Large \textcolor{green}{\textbf{$\alpha_1$}}};

			\draw ({1.3*\RRR*cos(-24)},{1.3*\rrr*sin(-24)}) node {\Large \textcolor{green}{\textbf{$\alpha_2$}}};
			\draw ({1.3*\RRR*cos(-96)},{1.3*\rrr*sin(-96)}) node {\Large \textcolor{green}{\textbf{$\alpha_2$}}};

			\draw ({1.3*\RRR*cos(84)},{1.3*\rrr*sin(84)}) node {\Large \textcolor{blue}{\textbf{$\beta_1$}}};
			\draw ({1.3*\RRR*cos(12)},{1.3*\rrr*sin(12)}) node {\Large \textcolor{blue}{\textbf{$\beta_1$}}};
			
			\draw ({1.3*\RRR*cos(-60)},{1.3*\rrr*sin(-60)}) node {\Large \textcolor{blue}{\textbf{$\beta_2$}}};
			\draw ({1.3*\RRR*cos(-132)},{1.3*\rrr*sin(-132)}) node {\Large \textcolor{blue}{\textbf{$\beta_2$}}};

			\draw (-5,-3.6) node {\Large $H^+$};
			
			\draw (0,0) node {\Large $H^-$};
			
			\draw (0,\rrr) node {\Large $B_1^-$};
			
			
			\draw ({\RRR*cos(126)},{\rrr*sin(126)}) node {\Large $A_1^+$};
			\draw ({\RRR*cos(54)},{\rrr*sin(54)}) node {\Large $A_1^-$};
			\draw ({\RRR*cos(18)},{\rrr*sin(18)}) node {\Large $B_1^+$};
			\draw ({\RRR*cos(-18)},{\rrr*sin(-18)}) node {\Large $A_2^+$};
			\draw ({\RRR*cos(-54)},{\rrr*sin(-54)}) node {\Large $B_2^-$};
			\draw ({\RRR*cos(-90)},{\rrr*sin(-90)}) node {\Large $A_2^-$};
			\draw ({\RRR*cos(-126)},{\rrr*sin(-126)}) node {\Large $B_2^+$};


			\draw ({1.06*\RRR*cos(136)},{1.06*\rrr*sin(136)}) node {\small $1$};
			\draw ({1.06*\RRR*cos(115.5)},{1.06*\rrr*sin(115.5)}) node {\small $3$};

			\draw ({1.06*\RRR*cos(100)},{1.06*\rrr*sin(100)}) node {\small $4$};
			\draw ({1.06*\RRR*cos(80)},{1.06*\rrr*sin(80)}) node {\small $2$};

			\draw ({1.06*\RRR*cos(64)},{1.06*\rrr*sin(64)}) node {\small $3$};
			\draw ({1.06*\RRR*cos(44)},{1.06*\rrr*sin(44)}) node {\small $1$};
			
			\draw ({1.06*\RRR*cos(28.5)},{1.06*\rrr*sin(28.5)}) node {\small $2$};
			\draw ({1.06*\RRR*cos(7.5)},{1.06*\rrr*sin(7.5)}) node {\small $4$};

			\draw ({1.06*\RRR*cos(-7.5)},{1.06*\rrr*sin(-7.5)}) node {\small $5$};
			\draw ({1.06*\RRR*cos(-28.7)},{1.06*\rrr*sin(-28.7)}) node {\small $7$};

			\draw ({1.06*\RRR*cos(-43.6)},{1.06*\rrr*sin(-43.6)}) node {\small $8$};
			\draw ({1.06*\RRR*cos(-64.2)},{1.06*\rrr*sin(-64.2)}) node {\small $6$};

			\draw ({1.06*\RRR*cos(-80)},{1.06*\rrr*sin(-80)}) node {\small $7$};
			\draw ({1.06*\RRR*cos(-100)},{1.06*\rrr*sin(-100)}) node {\small $5$};

			\draw ({1.06*\RRR*cos(-115.2)},{1.06*\rrr*sin(-115.2)}) node {\small $6$};
			\draw ({1.06*\RRR*cos(-136)},{1.06*\rrr*sin(-136)}) node {\small $8$};

		\end{tikzpicture}

		\caption{Genus $2g+1$ Heegaard splitting of $\es^1$-bundle $M_e$ over $\Sigma_g$ for $e \leq 0$. The same numbers at the points where $\gamma$ intersects the circles $A^{\pm}_i$ and $B^{\pm}_i$ indicate the identified points.
		}\label{figure:heegaard_splitting_of_M_e}
	\end{figure}
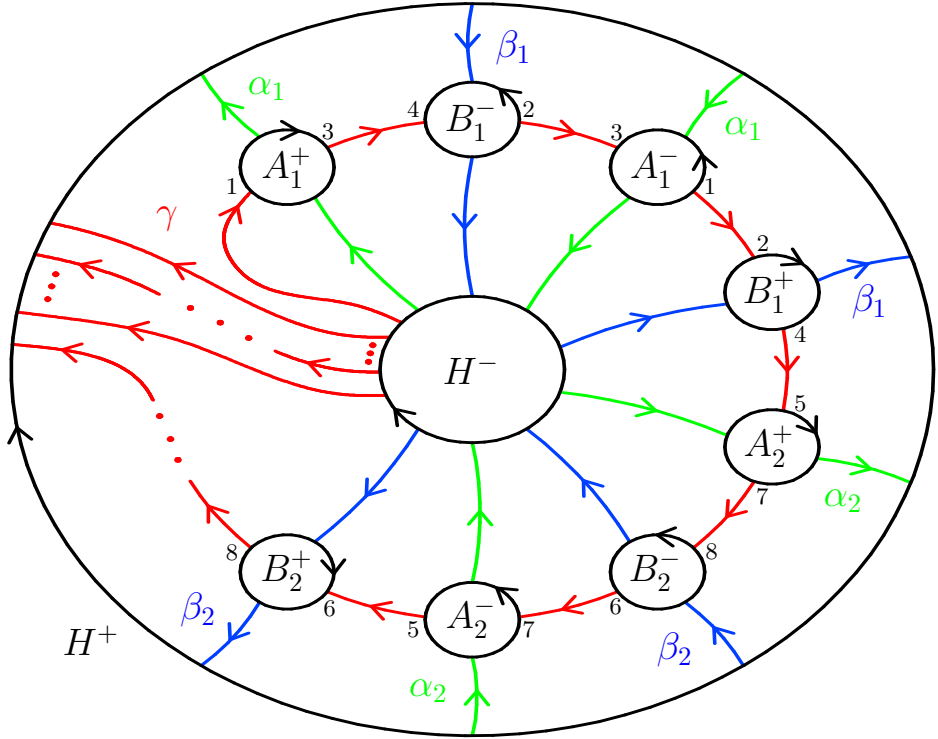
	
	Now, let us change the order of handles in the handle decomposition of $M_e$ given by this Heegaard splitting.  First, take a $0$-handle and $1$-handles corresponding to $H$ and $A_1,\ldots,A_g$. The obtained handlebody of genus $g+1$ with boundary $\Sigma'$ can be considered as in Figure \ref{figure:heegaard_splitting_of_M_e} after ignoring circles $B_i^+, B_i^-$. The curves $\alpha_i$ does not intersect $B_i$, so they lie in $\Sigma'$ and we attach $2$-handles to $\Sigma'$ corresponding to them resulting in a manifold~$W$. Note that each $\alpha_i$ splits the surface $\Sigma'$ into two components and so all the attached $2$-handles split $\Sigma'$ into $g+1$ components forming $\partial W$. Then, to these components we attached $g$ $1$-handles corresponding to $B_i$ in the figure obtaining a manifold with connected boundary $\Sigma_{g+1}$. Finally, we attach $g+1$ remaining $2$-handles corresponding to $\gamma$ and $\beta_i$, and one $3$-handle.
	
	Let $f\colon M_e \to \R$ be a Morse function corresponding to this handle decomposition. It can be simple by attaching the handles step by step. It is easy to see that $\beta_1(\reeb(f)) = g$. For example, $\reeb(f)$ has $2g$ vertices of degree $3$ which correspond to $g$ $2$-handles corresponding to $\alpha_i$ and $g$ $1$-handles corresponding to $B_i$ since these handles split or merge components of a surface they are attached. Moreover, $0$-handle and $3$-handle provide vertices of degree $1$ and the rest of $2g+2$ handles correspond to vertices of degree $2$. Thus $\Delta_3(\reeb(f)) = 2g$, $\Delta_2(\reeb(f))=2g+2$ and so $\beta_1(\reeb(f)) = -(1+1)/2+2g/2+1 = g$ by (\ref{formula:cycle_rank_in_terms_of_Delta_3_Lemma_DCG}). 
	
	For $e \neq \pm 1$ since $\Delta_2(\reeb(f)) = 2g+2 = 2(g(M_e) - \corank(\pi_1(M_e)))$, this is the minimal value equal to~$\Delta_2(M_e)$.
\end{proof}

It remains to investigate the case $e = \pm1$, for which $g(M_{\pm 1}) = 2g = \rank(\pi_1(M_{\pm 1}))$ by \cite{Zieschang}. Thus $2g+2 \geq \Delta_2(M_{\pm 1}) \geq 2g$ by the above proposition and bound~(\ref{formula:Delta_2>= 2(g(M)-corank)}).

Recall that a word $x^{\varepsilon_1}_{i_1}\ldots x^{\varepsilon_k}_{i_k}$, $\varepsilon_j = \pm 1$, in a free group $F_n=\geng{x_1,\ldots,x_n \,|\,}$ is \emph{cyclically reduced} if it is reduced and $x^{\varepsilon_k}_{i_k}x^{\varepsilon_1}_{i_1} \neq 1$. Note that if $r$ is a reduced word, then there exists an element $w \in F_n$ such that $wrw^{-1}$ is cyclically reduced and has the same normal closure as $r$. We will need the following result concerning with one-relator groups proved by W. Magnus \cite{Magnus1930} (cf. \cite{Magnus-book}). 

\begin{theorem}[\emph{Freiheitssatz}]
	If $r$ is a cyclically reduced word in a free group $F_n$ that contains $x_i$, then every non-trivial element of the normal closure of $r$ in $F_n$ also contains $x_i$.
\end{theorem}

%
%

We are now ready to prove the main theorem of this work.

\begin{proof}[Proof of Theorem \ref{theorem:main_theorem}]
	Consider the following rank $2g$ presentation $\P$ of $G=\pi_1(M_e)$, $e = \pm 1$,
	$$
	G \cong \P = \geng{a_1,b_1,\ldots,a_g,b_g \,|\, [a_i,h^e]=[b_i,h^e]=1},
	$$
	where $h=\Pi_i [a_i,b_i]$. To prove that $\Omega_{2g}(\pi_1(M_e)) = 2g$ it suffices to show that each relator in any rank $2g$ presentation of $G$ of deficiency $0$ is a word in all $2g$ generators in its reduced form.
	
	First, note that any two generating sets of $G$ of cardinality $2g$ are Nielsen equivalent (for a proof see \cite[A.1 Theorem]{Me_splittings}). Thus if
	$$
	G \cong \P' = \geng{x_1,y_1,\ldots,x_g,y_g \,|\, r_1, \ldots, r_{2g}}
	$$
	is another presentation of $G$, then there is an isomorphism $\varphi \colon F_{2g} \to F'_{2g}$, $F_{2g} = \geng{a_1,b_1,\ldots,a_g,b_g \,|\,}$ and $F'_{2g} = \geng{x_1,y_1,\ldots,x_g,y_g \,|\,}$, such that the diagram
	\begin{align*}
		\begin{xy}
			(0,15)*+{F_{2g}}="G2";
			(0,0)*+{F'_{2g}}="G1";
			(25,15)*+{\P}="F1";
			(25,0)*+{\P'}="F2";
			{\ar@{->}_{\pi'} "G1"; "F2"};%
			{\ar@{->}_{\varphi}^{\cong} "G2"; "G1"};%
			{\ar@{->}^{\pi} "G2"; "F1"};%
			{\ar@{->}^{}_{\cong}  "F1"; "F2"};
			(45,7.5)*+{G}="F";
			{\ar@{->}^{\cong} "F1"; "F"};%
			{\ar@{->}_{\cong} "F2"; "F"};%
		\end{xy}
	\end{align*}
	\noindent commutes, where $\pi$ and $\pi'$ are the canonical quotients. Thus $\varphi(\ker \pi) = \ker \pi'$.
	
	Since $h$ is central in $\P$, $\geng{h}$ is normal and the quotient $\P/\geng{h}$ is $\pi_1(\Sigma_g)$. If we denote by $\psi$ the quotient map $\P \to \pi_1(\Sigma_g)$ given by adding the relation $h=1$, then $\ker \pi \subset \ker (\psi\circ \pi) = \geng{h}^{F_{2g}}$, the normal closure of $h$ in $F_{2g}$. Moreover, $\ker(\pi') \subset \varphi(\ker (\psi \circ \pi)) = \geng{\varphi(h)}^{F'_{2g}}$. Obviously, $F'_{2g}/\geng{\varphi(h)}^{F'_{2g}} \cong F_{2g}/\geng{h}^{F_{2g}} \cong  \pi_1(\Sigma_g)$. Note that $h$ is a cyclically reduced word containing each generator of $F_{2g}$. If $\varphi(h)$ does not contain some generator of $F'_{2g}$ in its cyclically reduced form, say $x$, then $\pi_1(\Sigma_g) \cong F'_{2g}/\geng{\varphi(h)}^{F'_{2g}}$ is a non-trivial free product, a contradiction. Thus $\varphi(h)$ contains every generator of $F'_{2g}$.
	
	To sum up, we showed that for any rank $2g$ presentation of $G$ its relators are elements of the normal closure of an element having all $2g$ generators of a free group in its cyclically reduced form. Thus by the \emph{Freiheitssatz} every such relator is a word in all $2g$ generators. Therefore $\Omega_{2g}(\pi_1(M_e))=2g$.
	
	By Theorem \ref{theorem:simplicity_cycles_and_degree_2_vertices} for any simple Morse function $f\colon M_e \to \R$ with the minimum number of critical points, so $k_1 = g(M_e) = 2g$ and $k_0 = k_3 = 1$, we have
	$$
	\beta_1(\reeb(f)) \leq g(M_e) - \Omega_{g(M_e)}(\pi_1(M_e)) = 2g - 2g = 0.
	$$
	Thus $\reeb(f)$ is a tree and $\Delta_2(\reeb(f)) = 4g$.
	
	The construction of the desired function with two more critical points, so $k_1 = g(M_e)+1 = 2g+1$, is presented in Proposition \ref{proposition:Heegaard_splitting_of_M_e}.
\end{proof}

\section{Final remarks}\label{section:final_remarks} 

The behaviour of $\Delta_2(M)$ with respect to the connected sum operation is unclear. However, we predict it is additive for  orientable $3$-manifolds.

\begin{problem}\label{theorem:Delta_2_is_additive_for_connected_sum}
	Is it true that for closed orientable $3$-manifolds $M$ and $N$ we have
	$$\Delta_2(M \# N) = \Delta_2(M) + \Delta_2(N)?$$
\end{problem}

\begin{remark}
	Let $\widetilde{M} = M \setminus \int D^3$. As we know, having two simple Morse functions $f_i \colon M_i \to \R$ and deleting a neighbourhood of a minimum of one function and a maximum of the second function, we produce functions on $\widetilde{M_i}$ which after possible translation can be pasted together to a simple Morse function $f = f_1 \# f_2$ on $M_1 \# M_2$. The constructed function $f$ has a sphere $\es^2$ as a connected component of its level set which separates $M_1 \# M_2$ into two components $\widetilde{M_i}$.
	
	However, it is easy to construct a function on $M_1 \# M_2$ which cannot be split in this way. It suffices to appropriately rearrange critical points of $f_1 \# f_2$.
\end{remark}

By Lemma \ref{lemma:additivity_for_connected_sum_when_equality_in_lower_bounds} the additivity of $\Delta_2(M)$ under the connected sum is related to equality in its lower bounds. Corollary \ref{corollary:at_least_4_degree_2_vertices_if_genus_and_corank_differs_by_1} shows that the inequality (\ref{formula:Delta_2>= 2(g(M)-corank)}) is strict if $M$ is an irreducible $3$-manifold such that $g(M) = \corank(\pi_1(M))+1 \geq 2$. However, for known examples the non-negative number
$$
\frac{1}{2}\Delta_2(M) + \corank(\pi_1(M)) - g(M)
$$
admits values $0$ and $1$. Can it be greater than $1$ in the case of irreducible $3$-manifolds?

Another question concerning Reeb graphs of functions on $3$-manifolds is about their relations with geometries. In particular, how can we find simple Morse functions on a hyperbolic $3$-manifold $M$ whose Reeb graphs have cycle rank equal to $\corank(\pi_1(M))$? Does it depend on some geometrical properties? It may be related to computations of the invariant $\Omega(G)$ for hyperbolic groups.

\address

\end{document}